\documentclass[11pt]{article}
\usepackage{amsmath, amscd, amssymb, amsthm}
\usepackage{anysize}
\usepackage{mathrsfs}
\usepackage[displaymath, mathlines]{lineno}
\usepackage{enumitem}
\usepackage{color}
\usepackage{tikz}

\usetikzlibrary{arrows.meta}

\marginsize{1in}{1in}{1in}{1in}

\newtheorem{theorem}{Theorem}[section]
\newtheorem{proposition}[theorem]{Proposition}
\newtheorem{corollary}[theorem]{Corollary}
\newtheorem{lemma}[theorem]{Lemma}
\newtheorem{example}{Example}
\newtheorem{observation}[theorem]{Observation}

\newtheorem{problem}{Problem}

\newcommand{\mof}{\mathrm{mof}}
\newcommand{\MOF}{\mathrm{MOF}}
\newcommand{\mim}{\mathrm{mim}}

\newcommand{\diam}{{\rm diam}}

\definecolor{DarkGreen}{rgb}{0.2, 0.6, 0.3}

\definecolor{electricindigo}{rgb}{0.44, 0.0, 1.0}

\newenvironment{enum-roman}{\begin{enumerate}[label=(\roman*), leftmargin=2\parindent]}{\end{enumerate}}

\newenvironment{enum-alph}{\begin{enumerate}[label=(\alph*), leftmargin=2\parindent]}{\end{enumerate}}

\title{Extremal $k$-forcing sets in oriented graphs}
\author{Yair Caro\thanks{Department of Mathematics and Physics, University of Haifa-Oranim, Israel}, Randy Davila\thanks{Department of Pure and Applied Mathematics, University of Johannesburg, South Africa}, and Ryan Pepper\thanks{Department of Mathematics and Statistics, University of Houston--Downtown, USA}}%\author[D. Amos]{David Amos}
\begin{document}

\color{black}

\maketitle

\begin{abstract}
This article studies the \emph{$k$-forcing number} for oriented graphs, generalizing both the \emph{zero forcing number} for directed graphs and the $k$-forcing number for simple graphs. In particular, given a simple graph $G$, we introduce the maximum (minimum) oriented $k$-forcing number, denoted $\MOF_k(G)$ ($\mof_k(G)$), which is the largest (smallest) $k$-forcing number among all possible orientations of $G$. These new ideas are compared to known graph invariants and it is shown that, among other things, $\mof(G)$ equals the path covering number of $G$ while $\MOF_k(G)$ is greater than or equal to the independence number of $G$ -- with equality holding if $G$ is a tree or if $k$ is at least the maximum degree of $G$. Along the way, we also show that many recent results about $k$-forcing number can be modified for oriented graphs.
\end{abstract}

\section{Introduction and basic definitions}
In this paper we discuss the \emph{$k$-forcing number} of an oriented graph. This concept generalizes the \emph{directed zero forcing number}, first introduced in \cite{Hogben2010} and studied in \cite{MinRankDigraphs}, while also expanding recent work on the $k$-forcing number introduced in \cite{Amos2014} and studied further in \cite{Dynamic}. The notion of zero forcing (for simple graphs) was introduced independently in \cite{AIM} and \cite{Burgarth}. In \cite{AIM}, it was introduced to bound from below the minimum rank of a graph, or equivalently, to bound from above the maximum nullity of a graph. Namely, if $G$ is a graph whose vertices are labeled from $1$ to $n$, then let $M(G)$ denote the maximum nullity over all symmetric real valued matrices where, for $i\neq j$, the ij$^{th}$ entry is nonzero if and only if $\{i,j\}$ is an edge in $G$. Then, the zero forcing number is an upper bound on $M(G)$. In \cite{Burgarth}, it is indirectly introduced in relation to a study of control of quantum systems. Additionally, the problem of zero forcing number is closely related to the Power Dominating Set Problem, which is motivated by monitoring electric power networks using Kirchoff's Law \cite{power}.  One can also imagine other applications in the spread of opinions or disease in a social network (as also described for a similar invariant by Dreyer and Roberts in \cite{Dreyer}). Many other papers have been written about zero forcing and its variants in recent years (for example \cite{Chilakammari, DaHe17+, Davila Kenter, Yi}). While most of the first papers written were from a linear algebra point of view (\cite{MinRankDigraphs, Edholm, Row2, Row}), a fruitful change to a graph theoretic approach, and connection to basic graph theoretic parameters such as degree and connectivity, as well as the more general notion of $k$-forcing, was introduced and developed in \cite{Amos2014} and \cite{Dynamic}. The main point of this paper is the introduction of the new invariants, $\MOF_k(G)$ and $\mof_k(G)$, for a simple graph $G$, that represent the extremal cases of the $k$-forcing number over all possible orientations of $G$. Before proceeding, we will give the basic terminology needed.

Let $G$ be a simple, finite, and undirected graph with vertex set $V(G)$ and edge set $E(G)$. The order and size of $G$ will be denoted $n = |V(G)|$ and $m = |E(G)|$, respectively. A graph with order $1$ is called a \emph{trivial graph}. If $E(G) = \emptyset$, we say that $G$ is the \emph{empty graph}; otherwise $G$ is a \emph{non-empty graph}.Two vertices $u$ and $w$ in $G$ are said to be \emph{adjacent}, or \emph{neighbors}, if $\{u,v\}\in E(G)$. The \emph{open neighborhood} of a vertex $v$ in $G$, is the set of neighbors of $v$, written $N(v)$, whereas the \emph{closed neighborhood} of $v$ is $N[v] = N(v)\cup\{v\}$. The degree of a vertex $v\in V(G)$, written $d(v)$, is the number of neighbors of $v$ in $G$; and so, $d(v) = |N(v)|$. The minimum degree, average degree, and maximum degree of $G$ will be denoted $\delta(G)$, $d(G)$, and $\Delta(G)$, respectively. If the graph $G$ is clear from the context, we simply write $V$, $E$, $n$, $m$, $\delta$, and $\Delta$, rather than $V(G)$, $E(G)$, $n(G)$, $m(G)$, $\delta(G)$, and $\Delta(G)$, respectively. A graph $G$ is \emph{connected} if for all vertices $v$ and $w$ in $G$, there exists a $(v,w)$-path. The length of a shortest $(v,w)$-path in $G$, is the distance between $v$ and $w$, and is written $d(v,w)$. The \emph{diameter} of $G$, written $\diam(G)$, is the maximum distance among all pairs of vertices in $G$. A set of vertices $I\subseteq V(G)$ is \emph{independent} if the vertices of $I$ are pairwise non-adjacent. The cardinality of a maximum independent set in $G$, is the \emph{independence number} of $G$, and is denoted $\alpha(G)$. For notation and terminology not defined here, the reader is referred to \cite{West}.
 
An \emph{oriented graph} (also called an \emph{asymmetric digraph}) $D$, is a digraph that can be obtained from a graph $G$ by assigning to each edge $\{u,v\} \in E$ exactly one of the ordered pairs $(u,v)$ and $(v,u)$ (orienting the edges)-- which we call \emph{arcs}. We call the resulting digraph $D$ an \emph{orientation} of $G$, and say that $D$ is an oriented graph with underlying graph $G$. Let $D$ be an oriented graph with underlying simple graph $G$. If $(u,v)$ is an arc of $D$, then we say that $u$ is \emph{directed towards} $v$, that $v$ is an \emph{out-neighbor} of $u$, and that $u$ is an \emph{in-neighbor} of $v$. For any vertex $v$ of $D$, the \emph{out-degree} (resp. \emph{in-degree}) of $v$ is denoted by $d^+(v)$ (resp. $d^-(v)$), and is the number of out-neighbors of $v$ (resp. in-neighbors of $v$). The \emph{minimum out-degree} (resp. in-degree) is denoted $\delta^+ = \delta^+(D)$ (resp. $\delta^- = \delta^-(D)$), and the maximum out-degree (resp. in-degree) is denoted $\Delta^+ = \Delta^+(D)$ (resp. $\Delta^- = \Delta^-(G)$). If every vertex has the same out-degree (resp. in-degree), then $D$ is said to be \emph{out-regular} (resp. \emph{in-regular}). A \emph{directed path} in $D$ is a sequence of vertices $u_1, u_2, \ldots, u_p$ of $D$ such that $(u_i, u_{i+1})$ is an arc of $D$, $1 \leq i \leq p-1$. We say that the directed path begins at $u_1$ and ends at $u_p$, and that $u_p$ is \emph{reachable} from $u_1$. The vertices $u_1$ and $u_p$ are called the \emph{end-vertices} of the directed path, and the directed path has \emph{length} $p$.  An oriented graph $D$ is called a \emph{reachable oriented graph}  if there exists a vertex $v$ such that every other vertex of $D$ is reachable from $v$. We call $v$ a \emph{root} of $D$. We say that $D$ is a \emph{strongly reachable oriented graph} (sometimes called a strongly connected graph) if for every pair of vertices $u,v$ of $D$, $u$ is reachable from $v$ and $v$ is reachable from $u$. Finally, throughout the entire paper, $k$ is assumed to be a positive integer. 

Now we will describe the $k$-forcing process for oriented graphs. Suppose that $D$ is an orientation of $G$, and the vertices of $D$ are colored and non-colored, with at least one vertex being colored. For each positive integer $k$, we define the \emph{$k$-color change rule} as follows: any colored vertex that is directed towards at most $k$ non-colored vertices (has at most $k$ non-colored out-neighbors) forces each of these non-colored neighbors to become colored. A colored vertex that forces a non-colored vertex to become colored is said to \emph{$k$-force} that vertex to become colored. Let $S$ be any nonempty subset of vertices of $D$. By the \emph{oriented $k$-forcing process starting from $S$}, we mean the process of first coloring the vertices of $S$, while $V(D) \setminus S$ remains non-colored, and then iteratively applying the $k$-color change rule as many times as possible, until no further color changes occur. During each step (or iteration) of the oriented $k$-forcing process, all vertices that $k$-force do so simultaneously. If after termination of the oriented $k$-forcing process, every vertex of $D$ is colored, we say that $S$ is a \emph{oriented $k$-forcing set} (or simply a \emph{$k$-forcing set}) for $D$. The cardinality of a smallest oriented $k$-forcing set for $D$ is called the \emph{oriented $k$-forcing number} of $D$ and is denoted $F_k(D)$. When $k=1$, we will drop the subscript from our notation and write $F(D)$ instead of $F_1(D)$, and this case corresponds to the directed zero forcing number (typically denoted $Z(D)$). The maximum oriented $k$-forcing number over all orientations of $G$ is denoted $\MOF_k(G)$, whereas the minimum oriented $k$-forcing number over all orientations of $G$ is denoted $\mof_k(G)$. When $k=1$, we will drop the subscripts and write $\MOF(G)$ and $\mof(G)$.

The remainder of the paper is organized as follows. In Section \ref{sec: dkf} we discuss some basic results for the oriented $k$-forcing number, many of which are extensions of the results in \cite{Amos2014} and \cite{Dynamic} to oriented graphs.
In Section \ref{sec: ekf}, and its subsections, we study $\mof_k(G)$ and $\MOF_k(G)$. In Section \ref{conclusion}, we offer some concluding remarks and state future areas of research and open problems.

%%%%%%%%%%%%%%%%%%%%%%%%%%%%%%%%%%%%%%%%%%%%%%%%%%%%%%%%%%%%%%%%%%%%%%%%%%%%%%%%%%%%%%%%%%%%%%%%%%%%%%%%%%%%%%%%%%%%%%%%%%%

\section{Oriented $k$-forcing}\label{sec: dkf}

\subsection{Basic results}\label{sec: dkf basic results}

In this section, we collect some basic results for the $k$-forcing number of an oriented graph, many of which will be useful in what follows.  Since any subset of vertices in an oriented graph $D$ that contains a $k$-forcing set is also $k$-forcing, we have the following result.

\begin{observation}\label{subset forcing}
Let $D$ be an oriented graph. If $S$ is a $k$-forcing set of $D$ and $S$ is contained in $T$, then $T$ is a $k$-forcing set of $D$.
\end{observation}

Interestingly, the $k$-forcing number is not monotonic with respect to induced subgraphs. As the next example shows, there are induced subgraphs which have a larger $k$-forcing number than the graph itself.

\begin{figure}[htb]\label{fig1}
\begin{center}
\begin{tikzpicture}[scale=.8,style=thick,x=1cm,y=1cm, =>stealth]
\def\vr{2.5pt} % \vr = vertex radius;
% define vertices
%
\path (-3,1) coordinate (x1);
\path (-2,1) coordinate (x2);
\path (-1,1) coordinate (x3);
\path (0,1) coordinate (x4);
\path (1,1) coordinate (x5);
\path (2,1) coordinate (x6);

\path (0,-.5) coordinate (v);

%\path (0.75,1.3) coordinate (w);
%
%  edges
\draw (x1) -- (x2);
\draw (x2) -- (x3);
\draw (x3) -- (x4);
\draw (x4) -- (x5);
\draw (x5) -- (x6);

\draw (x2) -> (v);
\draw (x4) -- (v);
\draw (x6) -- (v);
\draw (x1) [fill=white]circle (\vr);
\draw (x2) [fill=white]circle (\vr);
\draw (x3) [fill=white] circle (\vr);
\draw (x4) [fill=white] circle (\vr);
\draw (x5) [fill=white] circle (\vr);
\draw (x6) [fill=white] circle (\vr);

\draw (v) [fill=white] circle (\vr);

\draw (0,-1.5) node {The graph $G_6$};
\draw[anchor = south] (x1) node {{\small $x_1$}};
\draw[anchor = south] (x2) node {{\small $x_2$}};
\draw[anchor = south] (x3) node {{\small $x_3$}};
\draw[anchor = south] (x4) node {{\small $x_4$}};
\draw[anchor = south] (x5) node {{\small $x_5$}};
\draw[anchor = south] (x6) node {{\small $x_6$}};
\draw[anchor = north] (v) node {{\small $v$}};

%%%%%%%%%
\path (4.5,1) coordinate (x1);
\path (5.5,1) coordinate (x2);
\path (6.5,1) coordinate (x3);
\path (7.5,1) coordinate (x4);
\path (8.5,1) coordinate (x5);
\path (9.5,1) coordinate (x6);

\path (7.5,-.5) coordinate (v);

\draw[black, arrows={->[line width=1pt,black,length=3mm,width=3mm]}] (x1) -- (x2);
\draw[black, arrows={->[line width=1pt,black,length=3mm,width=3mm]}] (x2) -- (x3);
\draw[black, arrows={->[line width=1pt,black,length=3mm,width=3mm]}] (x3) -- (x4);
\draw[black, arrows={->[line width=1pt,black,length=3mm,width=3mm]}] (x4) -- (x5);
\draw[black, arrows={->[line width=1pt,black,length=3mm,width=3mm]}] (x5) -- (x6);

\draw[black, arrows={->[line width=1pt,black,length=4.25mm,width=3.5mm]}] (x2) -- (v);
\draw[black, arrows={->[line width=1pt,black,length=4.25mm,width=3.5mm]}] (x4) -- (v);
\draw[black, arrows={->[line width=1pt,black,length=4.25mm,width=3.5mm]}] (x6) -- (v);

\draw (x2) -- (x4);

\draw (x1) [fill=black] circle (\vr);
\draw (x2) [fill=white] circle (\vr);
\draw (x3) [fill=white] circle (\vr);
\draw (x4) [fill=white] circle (\vr);
\draw (x5) [fill=white] circle (\vr);
\draw (x6) [fill=white] circle (\vr);

\draw (v) [fill=black] circle (\vr);

\draw[anchor = south] (x1) node {{\small $x_1$}};
\draw[anchor = south] (x2) node {{\small $x_2$}};
\draw[anchor = south] (x3) node {{\small $x_3$}};
\draw[anchor = south] (x4) node {{\small $x_4$}};
\draw[anchor = south] (x5) node {{\small $x_5$}};
\draw[anchor = south] (x6) node {{\small $x_6$}};

\draw[anchor = north] (v) node {{\small $v$}};
\draw (8,-1.5) node {The orientation $D_6$};
%%%%%%%%%
\end{tikzpicture}
\end{center}
\vskip -0.5 cm
\caption{The graph $G_6$ and the orientation $D_6$ illustrating Example~\ref{not monotone subgraph}.} \label{f:fig1}
\end{figure}
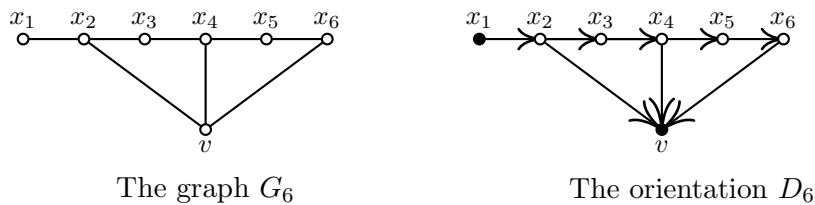

\begin{example}[Showing oriented forcing is not monotonic with induced subgraphs]\label{not monotone subgraph}
Take a path on $p \geq 6$ vertices labeled $\{x_1,x_2,\ldots,x_p\}$, where $p$ is an even number, together with an additional vertex $v$ joined to every vertex of the path with even index, $\{x_2,x_4,\ldots,x_p\}$. Call this graph $G_p$ and let $D_p$ be the orientation obtained by directing $x_i$ to $x_{i+1}$ for every $i$ satisfying $1 \leq i \leq p-1$. Now direct every vertex of the path adjacent to $v$ towards $v$. Let $H$ denote the oriented subgraph induced by the initial path and let $K$ denote the star, or $K_{1,\frac{p}{2}}$, induced by  the set of vertices with even index together with $v$, namely $\{v,x_2,x_4,\ldots,x_p\}$. Now, for $G_p$, the $1$-forcing number is $2$ and $\{x_1,v\}$ is a minimum $1$-forcing set. For $H$, the $1$-forcing number is $1$ and the set $\{x_1\}$ is a minimum $1$-forcing set. For $K$, the $1$-forcing number is $\frac{p}{2}$ and the set $\{x_2,x_4,\ldots,x_p\}$ is a minimum $1$-forcing set. Thus, $F(H) < F(G_p) < F(K)$, so there can be induced subgraphs with both larger and smaller forcing numbers, for a given orientation, than the graph itself. See Figure 1 for an illustration.
\end{example}

We generally focus on connected graphs since the $k$-forcing number is additive across components. 

\begin{observation}\label{additive}
If $G$ is a graph with connected components $G_1,\dots, G_q$, and $D$ is an orientation of $G$, with respective induced oriented subgraphs $D_1,\dots, D_q$, then $F_k(D) =  \sum_{i=1}^q F_k(D_i)$.
\end{observation} 

Our next result is that the oriented $k$-forcing number is monotonic with $k$.

\begin{proposition}\label{monotonic}
If $D$ is an oriented graph, then $F_k(D) \geq F_{k+1}(D)$.
\end{proposition}
\begin{proof}
Let $D$ be an oriented graph, $k$ be a positive integer, and $S$ be a minimum $k$-forcing set of $D$. Each vertex that $k$-forces its neighbors to change color during the $k$-forcing process starting with $S$, will also $(k+1)$-force its neighbors to change color, since having at most $k+1$ non-colored neighbors is implied by having at most $k$ non-colored neighbors. Thus, the $(k+1)$-forcing process starting with $S$ will color every vertex $D$, at least as quickly as the $k$-forcing process starting with $S$. Therefore, $S$ is a $(k+1)$-forcing set of $D$ and $F_{k+1}(D) \leq |S| = F_k(D)$, completing the proof.
\end{proof}

Since we see now that the $k$-forcing number decreases as $k$ increases, we look to the extreme cases.

\begin{proposition}\label{large-k}
If $D$ is a reachable oriented graph with maximum out-degree $\Delta^+$ with $k \geq \Delta^+$, then $F_k(D)=1$. 
\end{proposition}
\begin{proof}
Let $k$ be a positive integer with $k \geq \Delta^+$. Since $D$ is reachable, there is a vertex in $D$ such that all other vertices of $D$ are reachable by directed paths from that vertex. Let $v$ be such a vertex and color $v$, leaving the rest of the vertices of $D$ non-colored. Since $v$ has at most $\Delta^+$ out-neighbors, it will $k$-force all of them to change color on the first step of the oriented $k$-forcing process. At each further step, every colored vertex has at most $\Delta^+$ out-neighbors and so will $k$-force them to change color. Since every vertex of $D$ is reachable from $v$ by a directed path, the $k$-forcing process will terminate when every vertex becomes colored. Hence $\{v\}$ is a oriented $k$-forcing set of $D$ and therefore, $F_k(D)=1$, as claimed.
\end{proof}

From these results, we arrive at the following chain of inequalities: 

\begin{equation*}
1= F_{\Delta^+}(D) \leq F_{\Delta^+ - 1}(D) \leq \dotsb \leq F_2(D) \leq F_1(D) = F(D).
\end{equation*} 

As a lower bound for the oriented $k$-forcing number, generalizing the fact that the zero forcing number is bounded below by the minimum degree, we give the following result. The proof is a only a slight modification of Proposition 2.1 from \cite{Amos2014}, and so, we omit it.

\begin{proposition}\label{Trivial Lower Bound}
If $D$ is an oriented graph with minimum out-degree $\delta^+$, then
\begin{equation*}
F_k(D) \geq  \max\{ \delta^+ - k + 1, 1 \}.
\end{equation*}
\end{proposition}

\subsection{$k$-Forcing chains, induced $k$-ary tree covers, and reversals}

Let $D$ be an oriented graph and let $S$ be a smallest $k$-forcing set for $D$. We construct a subgraph $F$ of $D$ as follows. First, remove all arcs of $D$. Then, for each $v \in V \setminus S$, add exactly one arc $(u,v)$ of $D$ where $u$ is a vertex that $k$-forces $v$ during some application of the $k$-color change rule during the oriented $k$-forcing process starting from $S$. This subgraph $F$ is a spanning forest. Moreover, for every vertex $v \in V \setminus S$, there exists a directed path in $F$ starting at some vertex in $S$ and ending at $v$. Hence the number of components of $F$ is at most $|S|$. Since no vertex of $S$ is $k$-forced during the $k$-forcing process (they are initially colored), there is no path in $F$ starting and ending at two different vertices in $S$. This, together with the manner in which arcs are included in $F$, implies that the number of components of $F$ is at least $|S|$. This means that $F$ has $|S| = F_k(D)$ components, and each vertex in $S$ is contained in precisely one component. Each component of $F$ is called a \emph{$k$-forcing chain}, and the set $\mathscr{F}$ of components of $F$ is called a set of \emph{$k$-forcing chains} for $S$. Observe that every vertex in a $k$-forcing chain has out-degree at most $k$ in that chain, each component of $\mathscr{F}$ is an oriented tree, and $\mathscr{F}$ covers all of the vertices of $D$.

A \emph{$k$-ary} tree is a rooted tree in which each vertex has at most $k$ children. For example, a $1$-ary tree is a directed path rooted at the end-vertex with out-degree $1$, and a $2$-ary tree is a binary tree. For any oriented graph $D$, an \emph{induced} $k$-ary tree is any subgraph that can be rooted at a vertex $v$ so that the subgraph is a $k$-ary tree, and $v$ is directed towards each of its children, which are all directed towards each of their children, and so on. The smallest number of vertex-disjoint induced $k$-ary trees that cover all of the vertices of $D$ is called the \emph{induced $k$-ary tree cover number} and is denoted $IT_k(D)$.

\begin{proposition} \label{prop: dkf induced k-ary tree}
If $D$ is an oriented graph with induced $k$-ary tree cover number $IT_k(D)$, then
\begin{equation}
\nonumber F_k(D) \geq IT_k(D).
\end{equation}
\end{proposition}

\begin{proof}
Let $S$ be a smallest $k$-forcing set for $D$ and $\mathscr{F}$ a set of $k$-forcing chains for $S$. Observe that each chain in $\mathscr{F}$ is an induced $k$-ary tree. Since $\mathscr{F}$ covers the vertices of $D$, $\mathscr{F}$ is an induced $k$-ary tree cover for $D$. Hence $F_k(D) = |\mathscr{F}| \geq IT_k(D)$.
\end{proof}

The \emph{reversal} of $D$ (sometimes called the \emph{converse} of $D$ \cite{Out-Domination}) and denoted $D'$, is the oriented graph with vertex set $V(D)$ such that $(u,v)$ is an arc of $D'$ if and only if $(v, u)$ is an arc of $D$. Next, we show that the $1$-forcing number is preserved under reversal. 

\begin{theorem}\label{Reversal Theorem}
If $D$ is an oriented graph with reversal $D'$, then
\begin{equation}
\nonumber F(D)=F(D').
\end{equation}
\end{theorem}
\begin{proof}
Let $S$ be a smallest $1$-forcing set for $D$ and $\mathscr{F}$ a set of $1$-forcing chains for $S$. Each $1$-forcing chain is a directed path rooted at the vertex in $S$, or an isolated vertex in $\mathscr{F}$. Let $S'$ be the set of end-vertices of each chain in $\mathscr{F}$ that are not in $S$, together with the isolated vertices of $\mathscr{F}$. Observe that $|S'| = |S|$. We claim that $S'$ is a $1$-forcing set for $D'$. Color the vertices of $S'$.

Let $u\in S'$ be a vertex which was forced on the final step of the $1$-forcing process on $D$, $P \in \mathscr{F}$ the directed path containing $u$, and $v$ the in-neighbor of $u$ in $P$ that forced $u$ during the last step of the $1$-forcing process on $D$. Suppose $u$ has an out-neighbor $w \neq v$ in $D'$. In this case, we claim that $w$ must be an end-vertex of one of the directed paths $Q \in \mathscr{F}$. If otherwise, it has an out-neighbor $z$ on $Q$ in $D$, which it could force during the $1$-forcing process on $D$ only after $u$ was forced by $v$ -- contradicting the fact that $u$ was forced on the last step. Thus, all out-neighbors of $u$ in $D'$, besides $v$, are in $S'$ and hence, are already colored. From this we conclude that $u$ forces $v$ on the first step of the $1$-forcing process on $D'$ starting from $S'$ -- having at most one non-colored out-neighbor in $D'$. Furthermore, since $u$ was arbitrarily chosen, this is also true of all vertices forced on the final step of the $1$-forcing on $D$. Therefore every vertex that is forced on the final step of the $1$-forcing process on $D$ starting with $S$, will force their out-neighbors on the directed paths of $\mathscr{F}$ to which they belong on the first step of the $1$-forcing process on $D'$ starting with $S'$. This argument may be repeated now for vertices forced in the second to last step of the $1$-forcing process on $D$ starting with $S$, forcing their corresponding out-neighbors on the directed paths of $\mathscr{F}$ to which they belong on the second step of the $1$-forcing process on $D'$ starting with $S'$. Continuing in this fashion, all of $D'$ is eventually colored by applying the $1$-forcing process starting with $S'$. Hence, $S'$ is a $1$-forcing set of $D'$ and $F(D') \leq |S'| = |S| = F(D)$. Finally, because $(D')'=D$, we have by the same argument applied to $D'$, $F(D)=F((D')') \leq F(D') \leq F(D)$. Therefore, these inequalities collapse and the theorem is proven.
\end{proof}

Theorem \ref{Reversal Theorem} fails for $k \geq 2$. To see this, consider $D=K_{1,n}$ with edges directed from the vertex of degree $n$ to the leaves. Here $F_k(D)=n-k$, but $F_k(D')=n-1$. 

\subsection{Upper bounds for the oriented $k$-forcing number}
Many of the results for $k$-forcing number can be modified to get similar results for oriented $k$-forcing, such as was the case with the results from Section \ref{sec: dkf basic results}. For more examples, the reader is referred to \cite{Amos2014}. The theorem below is a modification to oriented graphs of the main result from \cite{Dynamic}. It gives a tractable upper bound for the $k$-forcing number of a connected oriented graph. The proof is an oriented analogue of the greedy algorithm used in \cite{Dynamic}, included here to show how oriented $k$-forcing sets can be constructed. 

\begin{theorem} \label{thm: dkf upper bound 3}
If $D$ is a reachable oriented graph with order $n$, minimum out-degree $\delta^+$, and maximum out-degree $\Delta^+$ with $k \leq \Delta^+$, then
\begin{equation}
\nonumber F_k(D) \leq \frac{(\Delta^+ - k)n + k}{\Delta^+},
\end{equation}
and this inequality is sharp.
\end{theorem}

\begin{proof}
Let $v$ be a vertex such that every other vertex in $D$ is reachable from $v$ by a directed path (such a vertex must exist by the definition of reachable oriented graphs). Note that this implies $d^+(v) \geq 1$. Let $S$ be the set containing $v$ and precisely $\max\{0, d^+(v)-k\}$ out-neighbors of $v$. Color the vertices of $S$. Observe that $v$ will $k$-force its out-neighbors on the first step of the $k$-forcing process starting from $S$, since $v$ is directed towards at most $k$ non-colored vertices. Continue the $k$-forcing process as long as possible. If all of the vertices of $D$ become colored, then $S$ is a $k$-forcing set for $D$, and thus,
\begin{equation}
\nonumber F_k(D) \leq |S| = |\{v\}| + \max\{0, d^+(v)-k\} = \max\{1, d^+(v) - k + 1 \}.
\end{equation}
Otherwise, the oriented $k$-forcing process stops before all vertices of $D$ are colored. In this case, and since $D$ is reachable, and so, every vertex was reachable from $v$. In particular, there is at least one colored vertex $u \neq v$, which is directed towards at least $k + 1$ non-colored vertices. We next  greedily color the smallest number of vertices needed in order for $u$ to $k$-force its out-neighbors, allowing the oriented $k$-forcing process to continue. Let $a(u)$ be the number of out-neighbors of $u$ we need to color, in order for the oriented $k$-forcing process to continue. Note that $a(u) \leq d^+(u) - k$. Next greedily color the $a(u)$ out-neighbors of $u$. The proportion of vertices colored by us to the total number of vertices colored, whether by us or by the $k$-color change rule, is 
\begin{equation}
\nonumber \frac{a(u)}{a(u) + k} \leq \frac{d^+(u)-k}{d^+(u)}\leq \frac{\Delta^+-k}{\Delta^+},
\end{equation}
where both inequalities come from monotonicity.

Now, let the process continue as before, and each time the process stops before the entire vertex set is colored, iterate the above steps. Note that each stoppage requires coloring more vertices according to the proportion indicated in the upper bound above. This process stops in some finite number of steps, eventually coloring all of the vertices of $D$. Hence $S$, the initial set of colored vertices, together with the vertices colored each time the $k$-forcing process stops, is a $k$-forcing set for $D$. Call this set $S'$, and note that $F_k(D) \leq |S'|$. Next, observe that after $v$ $k$-forces on the first step of the $k$-forcing process on $D$ starting from $S$, at least $d^+(v) + 1$ vertices are colored. Hence at most $n - (d^+(v) + 1)$  vertices are non-colored, which means that the number of times that the above algorithm stops is at most $n - (d^+(v) + 1)$. Therefore, $S'$ satisfies the following inequality,
\[
|S'| \leq \Big(\frac{\Delta^+-k}{\Delta^+}\Big)(n-(d^+(v)+1)) + \max\{d^+(v) -k +1,1\}.
\]
This can be written as,
\[
|S'| \leq \frac{(\Delta^+-k)n}{\Delta^+} + \frac{\Delta^+ \max\{d^+(v) -k +1,1\} - (\Delta^+ - k)(d^+(v) + 1)}{\Delta^+}.
\]
Now, if $k \leq d^+(v)$, then $\max\{d^+(v) -k +1,1\} = d^+(v) -k +1$. This, together with the fact that $d^+(v) \leq \Delta^+$, allows us to simplify the above inequality as,
\[
|S'| \leq \frac{(\Delta^+-k)n + \Delta^+ (d^+(v) -k +1) - (\Delta^+ - k)(d^+(v) + 1)}{\Delta^+} \leq \frac{(\Delta^+-k)n + k}{\Delta^+}.
\]
On the other hand, if $k \geq d^+(v)$, then $\max\{d^+(v) -k +1,1\} = 1$. This, together with the fact that $k \leq \Delta^+$, allows us to simplify the above inequality as,
\[
|S'| \leq \frac{(\Delta^+-k)n + \Delta^+ - (\Delta^+ - k)(d^+(v) + 1)}{\Delta^+} \leq \frac{(\Delta^+-k)n + k}{\Delta^+}.
\]
Therefore, in either case, since $F_k(D) \leq |S'|$, the result follows. A family of graphs demonstrating that equality can hold is presented in Example \ref{ex1} below.

\end{proof}

The theorem above can easily be modified to account for oriented graphs that are not reachable, even when the underlying simple graph is connected. Recall that a reachable oriented graph is an oriented graph which possesses a vertex which can reach all other vertices by a directed path. For a given graph $G$, a vertex $v \in G$, and an orientation $D$ of $G$, the set of all vertices reachable from $v$ is called the component of $D$ reachable by $v$. A set of vertices $R$ is called a \textit{reaching} set of $D$ if every vertex of $D$ is reachable by a directed path from some vertex in $R$. A reaching set is minimal if no vertex in the set is reachable from any other vertex in the set. With this notation, a reachable oriented graph has a reaching set of order $1$. Given a minimal reaching set $R=\{v_1,v_2,\ldots,v_r\}$ of order $r$, the component of $v_1$ is the set of all vertices reachable from $v_1$, the component of $v_2$ is the set of all vertices reachable from $v_2$ that are not in $v_1$, and in general, the component of $v_i$ is the set of all vertices reachable from $v_i$ that are not in $v_j$ for any $j<i$. Now, $k$-foricng sets can be be constructed in each of these connected components according to the greedy algorithm in the proof above. The net effect of this is to add a coefficient of $r$ in front of $k$ in the numerator of the inequality from Theorem \ref{thm: dkf upper bound 3}. We summarize in the corollary below.

\begin{corollary}\label{components}
If $D$ is an oriented graph with order $n$, minimum out-degree $\delta^+$, maximum out-degree $\Delta^+$ with $k\leq \Delta^+$, and $R$ is a minimum reaching set of $D$ with order $r$, then
\begin{equation}
\nonumber F_k(D) \leq \frac{(\Delta^+ - k)n + rk}{\Delta^+}\le \frac{(\Delta^+ - 1)n + \alpha(G)k}{\Delta^+},
\end{equation}
where this inequality is sharp, and the right most inequality comes from a theorem of Gallai and Milgram asserting $|R|\le \alpha(G)$ (see~\cite{Gallai-Milgram}).
\end{corollary}

We can get a very similar result for strongly reachable oriented graphs. In fact, if $D$ is strongly reachable, we could start the greedy algorithm at any vertex, since every vertex is reachable by directed paths from every other. Starting the process with a vertex of minimum out-degree would lead to a smaller set, so using the same technique as described above, we get the following improvement of the inequality in Theorem \ref{thm: dkf upper bound 3} with the stronger hypothesis.

\begin{theorem}\label{main2}
If $D$ is a strongly reachable oriented graph with order $n$, minimum out-degree $\delta^+$, and maximum out-degree $\Delta^+$ with $k \leq \Delta^+$, then
\begin{equation}
\nonumber F_k(D) \leq \frac{(\Delta^+ - k)n + \max\{k(\delta^+ - \Delta^+ + 1), \delta^+(k - \Delta^+) + k\}}{\Delta^+} \leq \frac{(\Delta^+ - k)n + k}{\Delta^+}.
\end{equation}
\end{theorem}

Specifying that $k=1$ in Theorem \ref{thm: dkf upper bound 3} above, we get the nice simple corollary for the oriented $1$-forcing number (directed zero-forcing number) for reachable oriented graphs.

\begin{corollary}\label{simple_upper}
If $D$ is a reachable oriented graph with order $n$ and maximum out-degree $\Delta^+$, then,
\begin{equation}
\nonumber F(D) \leq \frac{(\Delta^+ - 1)n + 1}{\Delta^+},
\end{equation}
and this inequality is sharp.
\end{corollary}

Both Theorem \ref{thm: dkf upper bound 3} and Corollary \ref{simple_upper} are tight upper bounds in that the inequalities presented can be satisfied with equality. To see this, consider the following example, where the greedy algorithm described in the proofs above is exactly what is necessary to $k$-force the oriented graph.

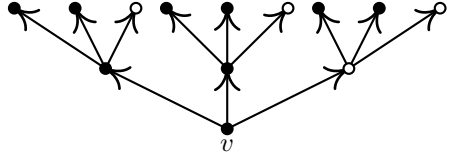
\begin{figure}[htb]
\begin{center}

\begin{tikzpicture}[scale=.8,style=thick,x=1cm,y=1cm]
\def\vr{2.5pt} % \vr = vertex radius;
% define vertices

\path (10.5,2) coordinate (x1);

\path (11.5,2) coordinate (u1);
\path (12.5,2) coordinate (u2);
\path (13,2) coordinate (u3);

\path (14,2) coordinate (x2);

\path (15,2) coordinate (u4);
\path (15.5,2) coordinate (u5);
\path (16.5,2) coordinate (u6);

\path (17.5,2) coordinate (x3);

\path (14,0) coordinate (v);
\path (12,1) coordinate (v1);
\path (14,1) coordinate (v2);
\path (16,1) coordinate (v3);
\draw[black, arrows={->[line width=1pt,black,length=3mm,width=3.5mm]}]  (v) -- (v1);
\draw[black, arrows={->[line width=1pt,black,length=3mm,width=3.5mm]}]  (v) -- (v2);
\draw[black, arrows={->[line width=1pt,black,length=3mm,width=3.5mm]}]  (v) -- (v3);

\draw[black, arrows={->[line width=1pt,black,length=3mm,width=3.5mm]}]  (v1) -- (u1);
\draw[black, arrows={->[line width=1pt,black,length=3mm,width=3.5mm]}]  (v1) -- (u2);
\draw[black, arrows={->[line width=1pt,black,length=3mm,width=3.5mm]}]  (v1) -- (x1);

\draw[black, arrows={->[line width=1pt,black,length=3mm,width=3.5mm]}]  (v2) -- (u3);
\draw[black, arrows={->[line width=1pt,black,length=3mm,width=3.5mm]}]  (v2) -- (u4);
\draw[black, arrows={->[line width=1pt,black,length=3mm,width=3.5mm]}]  (v2) -- (x2);

\draw[black, arrows={->[line width=1pt,black,length=3mm,width=3.5mm]}]  (v3) -- (u5);
\draw[black, arrows={->[line width=1pt,black,length=3mm,width=3.5mm]}]  (v3) -- (u6);
\draw[black, arrows={->[line width=1pt,black,length=3mm,width=3.5mm]}]  (v3) -- (x3);

\draw (v) [fill=black] circle (\vr);
\draw (v1) [fill=black] circle (\vr);
\draw (v2) [fill=black] circle (\vr);
\draw (v3) [fill=white] circle (\vr);
\draw (u1) [fill=black] circle (\vr);
\draw (u2) [fill=white] circle (\vr);
\draw (u3) [fill=black] circle (\vr);
\draw (u4) [fill=white] circle (\vr);
\draw (u5) [fill=black] circle (\vr);
\draw (u6) [fill=black] circle (\vr);
%\draw (14,-0.75) node {(c)};

\draw (x1) [fill=black] circle (\vr);
\draw (x2) [fill=black] circle (\vr);
\draw (x3) [fill=white] circle (\vr);

\draw[anchor = north] (v) node {{\small $v$}};

%\draw[anchor = south] (v) node {{\small $v$}};
%%%%%%%%%
\end{tikzpicture}
\end{center}
\vskip -0.7 cm
\caption{An illustration of Example~\ref{ex1}}
\end{figure}

\begin{example}\label{ex1}
Consider the oriented tree $T$, rooted at a vertex $v$ of out-degree $\Delta^+$, where each out-neighbor of $v$ has $\Delta^+$ additional out-neighbors, and each of those has $\Delta^+$ additional out-neighbors, and so on. Let $r$ denote the number of layers of this oriented tree, where the vertex $v$ is layer $0$, the $\Delta^+$ out-neighbors of $v$ are layer $1$, and so on. Furthermore, there are no other edges in the tree than those described, so that the last layer of leaves have out-degree equal to zero. For this tree, with $r \geq 1$, we use the geometric series formula to get,
\[
n=\sum_{i=1}^r (\Delta^+)^i = \frac{(\Delta^+)^{r+1}-1}{\Delta^+-1}.
\]
Since to $k$-force this tree, we need $v$, plus $\Delta^+-k$ out-neighbors of $v$, plus, for each of $\Delta^+$ out-neighbors, $\Delta^+-k$ out-neighbors, and so on for each layer. Thus, the $k$-forcing number of this oriented tree is given by,
\[
F_k(T)=1+(\Delta^+-k)+(\Delta^+-k)(\Delta^+) + (\Delta^+-k)(\Delta^+)^2+ \ldots + (\Delta^+-k)(\Delta^+)^{r-1}.
\]
This can be simplified again using the geometric series formula to get,
\[
F_k(T)=1+(\Delta^+-k)\sum_{i=0}^{r-1} (\Delta^+)^i = 1 + (\Delta^+-k)\frac{(\Delta^+)^r-1}{\Delta^+-1}.
\]
Finally, substituting these equations into Theorem \ref{thm: dkf upper bound 3}, or into Corollary \ref{simple_upper} if $k=1$, we see that equality holds.
\end{example}

We can get an improvement on Corollary \ref{simple_upper} when the reachable condition is replaced by the strongly reachable condition. Namely, since the reverse orientation $D'$ of a strongly reachable orientation $D$ is also strongly reachable, we can replace the role of $\Delta^+$ in Corollary \ref{simple_upper} by that of $\Delta^-$, and then appeal to 
Theorem \ref{Reversal Theorem} which states that $F(D)=F(D')$, to get the following result.

\begin{corollary}
If $D$ is a strongly reachable oriented graph with order $n$, maximum out-degree $\Delta^+$ and maximum in-degree $\Delta^-$, then
\begin{equation}
\nonumber F(D) \leq \min\Big{\{\frac{(\Delta^+ - 1)n + 1}{\Delta^+},\frac{(\Delta^- - 1)n + 1}{\Delta^-}}\Big\}.
\end{equation}
\end{corollary}

%%%%%%%%%%%%%%%%%%%%%%%%%%%%%%%%%%%%%%%%%%%%%%%%%%%%%%%%%%%%%%%%%%%%%%%%%%%%%%%%%%%%%%%%%%%%%%%%%%%%%%%%%%%%%%%%%%%%%%%%%%%

\section{Extremal $k$-forcing sets} \label{sec: ekf}

We now shift our focus to the extremal cases of the oriented $k$-forcing number over all orientations of a graph $G$, which is the second main section of our paper. Recall from the introduction that the \emph{maximum oriented $k$-forcing number}, denoted $\MOF_k(G)$, is defined as
\begin{equation}
\nonumber \MOF_k(G) = \max \{F_k(D) : D \textrm{ is an orientation of } G\},
\end{equation}
and that the \emph{minimum oriented $k$-forcing number} of $G$, denoted $\mof_k(G)$, is defined as
\begin{equation}
\nonumber \mof_k(G) = \min \{F_k(D) : D \textrm{ is an orientation of } G \}.
\end{equation}
Clearly then, $\mof_k(G) \leq F_k(D) \leq \MOF_k(G)$, for any orientation $D$ of the graph $G$. Recall also that when $k=1$, we drop the subscript from our notation and write $\mof(G)$ and $\MOF(G)$ for the minimum and maximum oriented $1$-forcing numbers, respectively. Before proceeding, we offer a simple example to illustrate the definitions and to highlight the differences.

\begin{example}\label{path}
Consider the path on $n$ vertices denoted $P_n=\{v_1,v_2,\ldots,v_n\}$. At one extreme, we can orient the edges so that $(v_i,v_{i+1})$ is an arc for every $i$ such that $1 \leq i \leq n-1$. In this case, $\{v_1\}$ is a $1$-forcing set since after it is colored, each vertex will force its unique out-neighbor along the directed path. Hence, $\mof(P_n)=1$. At the other extreme, if we orient the edges by including the arc $(v_1,v_2)$, and then repeatedly change the direction of each additional arc along the path from from $v_1$ to $v_n$, we create a situation where every other vertex has in-degree $0$ and therefore must be included in every $1$-forcing set. This turns out to be the orientation with the greatest possible forcing number for $P_n$ and $\MOF(P_n)=\lceil \frac{n}{2} \rceil$. This shows that the difference between $\MOF(G)$ and $\mof(G)$ can be arbitrarily large. See Figure~\ref{fig3} for an illustration.
\end{example}

\begin{figure}[htb]\label{fig3}
\begin{center}
\begin{tikzpicture}[scale=.8,style=thick,x=1cm,y=1cm, =>stealth]
\def\vr{2.5pt} % \vr = vertex radius;
% define vertices
%
\path (-3,0) coordinate (x1);
\path (-2,0) coordinate (x2);
\path (-1,0) coordinate (x3);
\path (0,0) coordinate (x4);
\path (1,0) coordinate (x5);
\path (2,0) coordinate (x6);

%\path (0.75,1.3) coordinate (w);
%
%  edges
\draw[black, arrows={->[line width=1pt,black,length=3mm,width=3mm]}] (x1) -- (x2);
\draw[black, arrows={->[line width=1pt,black,length=3mm,width=3mm]}] (x2) -- (x3);
\draw[black, arrows={->[line width=1pt,black,length=3mm,width=3mm]}] (x3) -- (x4);
\draw[black, arrows={->[line width=1pt,black,length=3mm,width=3mm]}] (x4) -- (x5);
\draw[black, arrows={->[line width=1pt,black,length=3mm,width=3mm]}] (x5) -- (x6);

\draw (x1) [fill=black]circle (\vr);
\draw (x2) [fill=white]circle (\vr);
\draw (x3) [fill=white] circle (\vr);
\draw (x4) [fill=white] circle (\vr);
\draw (x5) [fill=white] circle (\vr);
\draw (x6) [fill=white] circle (\vr);

\draw (0,-1.5) node {$\mof(P_6) = 1$};
\draw[anchor = south] (x1) node {{\small $x_1$}};
\draw[anchor = south] (x2) node {{\small $x_2$}};
\draw[anchor = south] (x3) node {{\small $x_3$}};
\draw[anchor = south] (x4) node {{\small $x_4$}};
\draw[anchor = south] (x5) node {{\small $x_5$}};
\draw[anchor = south] (x6) node {{\small $x_6$}};

%%%%%%%%%
\path (4.5,0) coordinate (x1);
\path (5.5,0) coordinate (x2);
\path (6.5,0) coordinate (x3);
\path (7.5,0) coordinate (x4);
\path (8.5,0) coordinate (x5);
\path (9.5,0) coordinate (x6);

\draw[black, arrows={->[line width=1pt,black,length=3mm,width=3mm]}] (x1) -- (x2);
\draw[black, arrows={->[line width=1pt,black,length=3mm,width=3mm]}] (x3) -- (x2);
\draw[black, arrows={->[line width=1pt,black,length=3mm,width=3mm]}] (x3) -- (x4);
\draw[black, arrows={->[line width=1pt,black,length=3mm,width=3mm]}] (x5) -- (x4);
\draw[black, arrows={->[line width=1pt,black,length=3mm,width=3mm]}] (x5) -- (x6);

\draw (x1) [fill=black] circle (\vr);
\draw (x2) [fill=white] circle (\vr);
\draw (x3) [fill=black] circle (\vr);
\draw (x4) [fill=white] circle (\vr);
\draw (x5) [fill=black] circle (\vr);
\draw (x6) [fill=white] circle (\vr);

\draw[anchor = south] (x1) node {{\small $x_1$}};
\draw[anchor = south] (x2) node {{\small $x_2$}};
\draw[anchor = south] (x3) node {{\small $x_3$}};
\draw[anchor = south] (x4) node {{\small $x_4$}};
\draw[anchor = south] (x5) node {{\small $x_5$}};
\draw[anchor = south] (x6) node {{\small $x_6$}};

\draw (8,-1.5) node {$\MOF(P_6) = 3$};
%%%%%%%%%
\end{tikzpicture}
\end{center}
\vskip -0.5 cm
\caption{An illustration of Example~\ref{path}.} \label{f:fig1}
\end{figure}
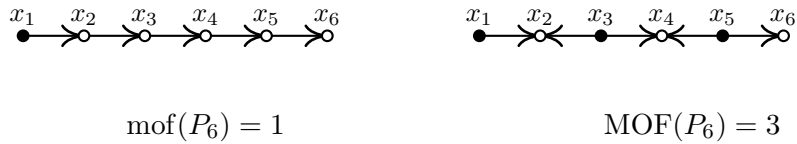

This section is organized as follows. In this Subsection \ref{sec: ekf basic facts}, we will establish some basic results about $\MOF_k(G)$ and $\mof_k(G)$. In Subsection  \ref {sec: mof}, we focus exclusively on $\mof_k(G)$, applying previous results and giving some new ones. In Subsection \ref{sec: MOF}, we focus exclusively on $\MOF_k(G)$, applying previous results and giving some new ones, including the value of this invariant for trees. 

\subsection{Basic results} \label{sec: ekf basic facts}

Since directed $k$-forcing was additive across components, the same is true for $\mof_k(G)$ and $\MOF_k(G)$. If $G$ is a disconnected graph with components $G_1, G_2, \ldots, G_r$, then both of the following equations are true;
\begin{equation}
\nonumber \mof_k(G) = \sum_{i=1}^r \mof_k(G_i),
\end{equation}

\begin{equation}
\nonumber \MOF_k(G) = \sum_{i=1}^r \MOF_k(G_i).
\end{equation}
Thus, it is enough to study the minimum and maximum oriented $k$-forcing numbers in connected graphs. As was shown in Proposition \ref{monotonic}, the oriented $k$-forcing number is monotonically non-increasing with $k$, that is $F_{k+1}(D) \leq F_k(D)$ for every positive integer $k$ and for every orientation $D$ of a graph $G$. This leads to the following observations.

\begin{observation}\label{monotonic_MOF}
If $G$ is a graph and $k$ is a positive integer, then the following are true,
\[
\mof_{k+1}(G) \leq \mof_k(G),
\]
\[
\MOF_{k+1}(G) \leq \MOF_k(G),
\]

\end{observation}

In Subsection \ref{sec: mof} we show that $\mof(G)=\rho(G)$, where $\rho(G)$ is the minimum number of vertex disjoint paths that cover all of the vertices of $G$ (called the path-covering number of $G$, see \cite{Gallai-Milgram, Hogben2010}). In Subsection \ref{sec: MOF} we show that $\MOF_\Delta(G) =\alpha(G)$. It is well known that $\rho(G) \leq \alpha(G)$ (since one can form an independent set by taking the end-vertices of the paths in a minimum path cover). These facts, together with Observation \ref{monotonic_MOF}, lead to the following result.

\begin{observation}
If $G$ is a non-empty graph with order $n$ and maximum degree $\Delta$ with $k \leq \Delta$, then the following chain of inequalities holds,
\[
\mof_{\Delta}(G) \leq \mof_{k}(G) \leq \mof(G) = \rho(G) \leq \alpha(G) = \MOF_{\Delta}(G) \leq \MOF_{k}(G) \leq \MOF(G) \leq n-1.
\]
\end{observation}

Furthermore, the last inequality, $\MOF(G) \leq n-1$, is sharp for complete bipartite graphs with all edges oriented from one part towards the other. 

%%%%%%%%%%%%%%%%%%%%%%%%%%%%%%%%%%%%%%%%%%%%%%%%%%%%%%%%%%%%%%%%%%%%%%%%%%%%%%%%%%%%%%%%%%%%%%%%%%%%%%%%%%%%%%%%%%%%%%%%%%%
\subsection{The minimum oriented $k$-forcing number -- $\mof_k(G)$} \label{sec: mof}

In this section, we focus our attention on $\mof_k(G)$. For what follows, we will use the term \emph{$(k+1)$-tree} to refer to tree with maximum degree at most $k+1$. A set of vertex-disjoint $(k+1)$-trees that cover the vertices of a graph $G$ is called a \emph{$(k+1)$-tree cover} for $G$. The cardinality of a smallest $(k+1)$-tree cover for $G$ is called the \emph{$(k+1)$-tree cover number} of $G$ and is denoted $T_k(G)$. Note that a $2$-tree is a path, and so $T_1(G)$ is the minimum number of vertex-disjoint paths needed to cover the vertices of $G$. In \cite{Hogben2010}, it is shown that the path cover number is a lower bound for the zero forcing number of a graph, $\rho(G) \leq F(G)=Z(G)$. As we shall see, the minimum oriented $1$-forcing number of a graph is equal to its path cover number.

Observe that if $T$ is a $(k+1)$-tree, then $T$ can be rooted at a leaf and oriented so that every vertex of $T$ has out-degree at most $k$. This is done by directing the root towards its neighbors, then each of these towards their neighbors (that aren't the root), and so on. In this case we say that $T$ is oriented away from the root. Conversely, any tree that can be rooted at a leaf and oriented away from the root so that each vertex has out-degree at most $k$ is a $(k+1)$-tree. Hence if $D$ is an orientation of graph $G$, $S$ a $k$-forcing set for $D$, and $\mathscr{F}$ a set of forcing chains for $S$, then each $k$-forcing chain in $\mathscr{F}$ is a $(k+1)$-tree. This implies that $F_k(D) \geq T_k(G)$ for all orientations $D$ of $G$ and hence $\mof_k(G) \geq T_k(G)$. In fact $\mof_k(G) = T_k(G)$, which is proven below.

\begin{theorem}\label{thm: mof_k = T_k}
For every graph $G$ and every positive integer $k$, $\mof_k(G) = T_k(G)$.
\end{theorem}
\begin{proof}
We have already noted, in the exposition before the statement of this theorem, that $\mof_k(G) \geq T_k(G)$, so we only need to show that $\mof_k(G) \leq T_k(G)$. Let $\mathscr{T} = \{T_1, T_2, \ldots, T_t\}$ be a smallest $k$-ary tree cover for $G$ (thus $t = |\mathscr{T}| = T_k(G)$). For each integer $p$ between $1$ and $t$, let $s_p$ be the root of $T_p$, and partition the vertices of each $T_p$ into levels as follows: level 0 contains $s_p$, level 1 contains all of the vertices whose distance from $s_p$ is 1, level 2 contains all of the vertices whose distance from $s_p$ is 2, and so on. Order the vertices in each level and let $v_{p,q,r}$ denote the $r^{th}$ vertex in the $q^{th}$ level of $T_p$. Note that $v_{p,0,1} = s_p$. 

Orient each tree $T_p \in \mathscr{T}$ away from the root $s_p$. Since $\mathscr{T}$ is a $(k+1)$-tree cover, all edges in $G$ that have not been oriented have the form $e = \{v_{p,q,r}, v_{p',q',r'}\}$ with $p \neq p'$. Direct these remaining edges as follows: if $q \leq q'$ direct $v_{p',q',r'}$ towards $v_{p,q,r}$, and if $q > q'$ direct $v_{p,q,r}$ towards $v_{p',q',r'}$. Call this orientation $D$ and let $S$ be the set of roots of $\mathscr{T}$, that is, $S = \{s_p = v_{p,0,1} : 1 \leq p \leq t\}$ (note that $|S| = T_k(G))$.

Color each vertex of $S$. Since $T_p$ is a $(k+1)$-tree rooted at $s_p$ for all $1 \leq p \leq t$, each $s_p \in S$ has at most $k$ out-neighbors in $T_p$ and all other out-neighbors are in $S$. Thus $s_p$ will $k$-force each non-colored out-neighbor in $T_p$ to become colored on the first application of the $k$-color change rule. Observe now that if $x = v_{p',q',r'}$ is an out-neighbor of $y = v_{p,q,r}$, where $p \neq p'$, then the index of the level containing $x$ in $T_{p'}$ is no more than the index of the level containing $y$ in $T_p$. Hence each out-neighbor $v$ of $s_p$ that was $k$-forced to become colored on the first application of the $k$-color change rule has at most $k$ non-colored out-neighbors -- namely the out-neighbors in $T_p$ -- and thus $v$ will $k$-force each of its non-colored out-neighbors to become colored on the second application of the $k$-color change rule. This continues until all of the vertices of $D$ become colored. Thus, $S$ is a $k$-forcing set for $D$, which means that $\mof_k(G) \leq F_k(D) \leq |S| = T_k(G)$.
\end{proof}

Theorem \ref{thm: mof_k = T_k} has several interesting corollaries, which we now list.

\begin{corollary} \label{cor: mof_k = 1 iff spanning k+1-tree}
For any graph $G$, $\mof_k(G) = 1$ if and only if $G$ has a spanning $(k+1)$-tree.
\end{corollary}

Taking $k=1$, a spanning $T_1(G)$ is a spanning $2$-tree cover of $G$, which is the same as a path cover. Thus, we arrive at the result below.

\begin{corollary}\label{pathcover}
If $G$ is a graph, then $\mof(G) = \rho(G)$.
\end{corollary}

A graph $G$ is \emph{path Hamiltonian} (also called \emph{traceable}) if there exists a path in $G$ that contains every vertex of $G$ exactly once. Hence, $G$ is path Hamiltonian if and only if $\rho(G)=1$. For references about the path cover number (also known as the path partition number) of a graph, and many more examples where $\mof(G)$ can be applied meaningfully, see \cite{path-partition}.

%%%%%%%%%%%%%%%%%%%%%%%%%%%%%%%%%%%%%%%%%%%%%%%%%%%%%%%%%%%%%%%%%%%%%%%%%%%%%%%%%%%%%%%%%%%%%%%%%%%%%%%%%%%%%%%%%%%%%%%%%%%

\subsection{Upper bounds for $\mof_k(G)$}\label{mof upper}

We now turn our attention to the presentation of some upper bounds on the minimum oriented $k$-forcing number. For this we keep in mind that, for any orientation $D$ of a simple graph $G$, the oriented $k$-forcing number is already an upper bound on $\mof_k(G)$.  

\begin{proposition}
If $G$ is a non-empty graph with maximum degree $\Delta$, then
\begin{equation}
\nonumber \mof_k(G) \leq \max \{n - k, n - \Delta \}.
\end{equation}
\end{proposition}

\begin{proof}
Let $D$ be an orientation of $G$ for which $\Delta^+(D) = \Delta$ (such an orientation is easily seen to exist by taking a vertex of maximum degree and orienting all of its incident edges towards its neighbors). Let $v$ be a vertex of $D$ with maximum out-degree, $d^+(v) = \Delta^+(D)$. Let $S$ be the set of all vertices of $D$ except for $\min \{k, \Delta\}$ out-neighbors of $v$. Color all of the vertices of $S$ and observe that $v$ will $k$-force its neighbors on the first step of the $k$-forcing process. Since all other vertices are colored initially, $S$ is a $k$-forcing set for $D$. Now,
\begin{equation}
\nonumber |S| = n - \min \{k, \Delta\} = \max \{n - k, n - \Delta\}.
\end{equation}
Therefore, $\mof_k(G) \leq F_k(D) \leq |S|$, completing the proof.
\end{proof}

A balanced orientation $D$ of a graph $G$ is one for which the in-degrees and out-degrees are equal, or only different by one, for every vertex. In particular, we say that $D$ is \emph{balanced} if $|d^+(v) - d^-(v)| \leq 1$ for every vertex $v$ of $D$. It is not hard to see that every graph has a balanced orientation (add a vertex and join it to every vertex of odd degree so that graph is Eulerian, orient the edges along the Eulerian trail, then remove the added vertex, see \cite{BalancedOrientation} for more). By considering a balanced orientation, we can use Corollary \ref{components} to bound $\mof(G)$ from above.

\begin{theorem}
If $G$ is a connected graph with order $n$, minimum degree $\delta\ge 2$, and maximum degree $\Delta$, then
\begin{equation}
\nonumber \mof(G) \leq \frac{ (\lceil \frac{\Delta}{2} \rceil -1) n + r}{ \lceil \frac{\Delta}{2} \rceil},
\end{equation}
where $r$ is the cardinality of a smallest reaching set over all balanced orientations of $G$.
\end{theorem}

\begin{proof}
Let $D$ be a balanced orientation for $G$ which has a reaching set of smallest order $r$. Then $|d^+(v) - d^-(v)| \leq 1$ for all vertices $v$ of $D$. Let $u$ be a vertex of $D$ such that $d(u)=d^-(u)+d^+(u)$ and $d^+(u) = \Delta^+(D)$. Note that $d^-(u) \geq 1$, since the minimum degree of $G$ is at least $2$, and the orientation $D$ is balanced. Thus, $\Delta^+(D) = d^+(u) \leq \lceil \frac{d(u)}{2} \rceil \leq \lceil \frac{\Delta}{2} \rceil$. Applying Corollary \ref{components} when $k=1$, we have
\begin{equation}
\nonumber F(D) \leq \frac{(\Delta^+(D) -1)n+r}{\Delta^+(D)} \leq \frac{(\lceil \frac{\Delta}{2} \rceil -1)n+r}{\lceil \frac{\Delta}{2} \rceil},
\end{equation}
where the second inequality follows from monotonicity. The theorem is now proven since we combine the above inequality with the fact that $\mof(G) \leq F(D)$.
\end{proof}

If $G$ has a reachable and balanced orientation, then we may replace $r$ by $1$ in the above theorem. Unfortunately, this is not always the case, as one can construct graphs with no reachable and balanced orientations. However, we make use of Robbin's Theorem, and its extension due to Nash-Williams (see \cite{well-balanced}), and observe that every $2$-edge connected graph has a reachable (connected) and balanced orientation. In light of this, we can simplify the above result with the following theorem.

\begin{theorem}
If $G$ is a $2$-edge connected graph with order $n$, and maximum degree $\Delta$, then
\begin{equation}
\nonumber \mof(G) \leq \frac{ (\lceil \frac{\Delta}{2} \rceil -1) n + 1}{ \lceil \frac{\Delta}{2} \rceil}.
\end{equation}
\end{theorem}

This shows, among other things, that for bridge-less cubic-graphs, $\mof(G) \leq \frac{n+1}{2}$. It makes it natural to ask if there is a constant $c <1$ such that $\mof_k(G) < cn$ for any graph $G$. Our next theorem says that the answer to this question is no, and we can have graphs where $\mof(G)$ is arbitrarily close to $n$.

\begin{example}\label{star}
Let $G=K_{1,n-1}$ be the star on $n\ge 2$ vertices. For this, graph, if $k<\Delta=n-1$, then it is not hard to see that $\mof_k(K_{1,n-1})=n-k-1$. This can be realized by orienting $k$ edges away from the central vertex and $\Delta-k$ edges towards the central vertex. Namely, if $v$ is the center, then $d^+(v)=k$ and $d^-(v)=\Delta -k$. The $k$-forcing number of this orientation is $\Delta -k=n-k-1$, and one can quickly see that any other orientation has $k$-forcing number at least this large. Now, for any constant $0<c<1$, one can chose $n$ large enough with respect to $k$ so that $\mof_k(K_{1,n-1})>cn$.
\end{example}

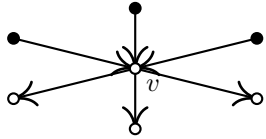
\begin{figure}[htb]
\begin{center}

\begin{tikzpicture}[scale=.8,style=thick,x=1cm,y=1cm]
\def\vr{2.5pt} % \vr = vertex radius;
% define vertices

\path (13,2) coordinate (u1);

\path (12,1.5) coordinate (v6);

\path (14,2) coordinate (v5);

\path (16,1.5) coordinate (v4);
\

\path (14,0) coordinate (v2);
\path (12,.5) coordinate (v1);
\path (14,1) coordinate (w);
\path (16,.5) coordinate (v3);
\draw[black, arrows={->[line width=1pt,black,length=3mm,width=3.5mm]}]  (w) -- (v1);
\draw[black, arrows={->[line width=1pt,black,length=3mm,width=3.5mm]}]  (w) -- (v2);
\draw[black, arrows={->[line width=1pt,black,length=3mm,width=3.5mm]}]  (w) -- (v3);
\draw[black, arrows={->[line width=1pt,black,length=3mm,width=3.5mm]}]  (v4) -- (w);
\draw[black, arrows={->[line width=1pt,black,length=3mm,width=3.5mm]}]  (v5) -- (w);
\draw[black, arrows={->[line width=1pt,black,length=3mm,width=3.5mm]}]  (v6) -- (w);

\draw (w) [fill=white] circle (\vr);

\draw (v1) [fill=white] circle (\vr);
\draw (v2) [fill=white] circle (\vr);
\draw (v3) [fill=white] circle (\vr);
\draw (v4) [fill=black] circle (\vr);
\draw (v5) [fill=black] circle (\vr);
\draw (v6) [fill=black] circle (\vr);

%\draw (14,-0.75) node {(c)};

\draw[anchor = north west] (w) node {{\small $v$}};

%\draw[anchor = south] (v) node {{\small $v$}};
%%%%%%%%%
\end{tikzpicture}
\end{center}
\vskip -0.7 cm
\caption{An illustration of Example~\ref{star} with $k=3$}
\end{figure}

%%%%%%%%%%%%%%%%%%%%%%%%%%%%%%%%%%%%%%%%%%%%%%%%%%%%%%%%%%%%%%%%%%%%%%%%%%%%%%%%%%%%%%%%%%%%%%%%%%%%%%%%%%%%%%%%%%%%%%%%%%%

\subsection{The maximum oriented $k$-forcing number -- $\MOF_k(G)$} \label{sec: MOF}

In this section we explore the maximum oriented $k$-forcing number of a graph $G$, denoted $\MOF_k(G)$. Basic properties of $\MOF_k(G)$ are discussed in Section \ref{sec: MOF basic properties}. In Section \ref{sec: MOF lower bounds}, we use these properties and other basic results to prove some general lower bounds for $\MOF_k(G)$. In Section \ref{trees} we discuss the maximum oriented $k$-forcing number for trees, which includes the interesting result that $\MOF(T)=\alpha(T)$ for every tree $T$.

We remark that in a forthcoming paper, the first and third author present a detailed study of the maximum oriented forcing number of complete graphs~\cite{MOF_Kn}. In particular, they show $\MOF(K_n)\ge \frac{3n-9}{4}$, and for $n\ge 2$, that $\MOF(K_n) \ge n - \frac{2n}{\log_2(n)}$. These results merit their own study, so we have omitted them from this paper. 

\subsubsection{Basic properties of $\MOF_k(G)$} \label{sec: MOF basic properties}

We begin by showing that, unlike $\mof(G)$ (recall Example~\ref{not monotone subgraph}) and the $k$-forcing number for a given oriented graph, the maximum oriented $k$-forcing number is monotone with respect to induced subgraphs. 

\begin{proposition} \label{prop: MOF monotonicity}
If $H$ is any induced subgraph of a graph $G$, then $\MOF_k(G) \geq \MOF_k(H)$.
\end{proposition}

\begin{proof}
Let $D$ be an orientation of $G$ for which the oriented subgraph induced by $H$ has $k$-forcing number equal to $\MOF_k(H)$, while every edge between $H$ and $V-H$ is oriented away from $H$ and towards $V-H$. Now, observe that the vertices of $H$ can only be $k$-forced by other vertices in $H$. Thus, any oriented $k$-forcing set of $D$ must included the vertices of $H$ needed to realize $\MOF_k(H)$, plus some maybe some additional vertices of $V-H$. Therefore, $\MOF_k(H) \leq F_k(D) \leq \MOF_k(G)$, which completes the proof.
\end{proof}

An edge $e$ of a connected graph $G$ is a \emph{bridge} (sometimes called a cut-edge) if the graph obtained from $G$ by removing $e$ has two components. We denote this graph by $G \setminus \{e\}$.

\begin{proposition}\label{bridge}
Let $G$ be a connected graph and $e \in E$ a bridge in $G$. If $G_1$ and $G_2$ are the components of $G \setminus \{e\}$, then $\MOF_k(G) \leq \MOF_k(G_1) + \MOF_k(G_2)$.
\end{proposition}

\begin{proof}
Let $D$ be any orientation for $G$ for which $F_k(D) = \MOF_k(G)$, and let $D_1$ and $D_2$ be the orientations for $G_1$ and $G_2$, respectively, inherited from $D$. Let $S_1$ and $S_2$ be smallest $k$-forcing sets for $D_1$ and $D_2$, respectively, and let $u \in G_1$ and $v \in G_2$ be the end-vertices of $e$. We assume without loss of generality that $u$ is directed toward $v$. Color the vertices in $S_1 \cup S_2$ black and all remaining vertices white. Observe that $S_2$ will $k$-force all of $D_2$ without being affected by the vertices in $D_1$. Once all of the vertices of $D_2$ have been colored black, $S_1$ will then $k$-force all of $D_1$. Thus $S_1 \cup S_2$ is a $k$-forcing set for $D$. It follows that,
\begin{equation}
\nonumber  \MOF_k(G) = F_k(D) \leq |S_1 \cup S_2| = |S_1| \cup |S_2| = F_k(D_1) + F_k(D_2) \leq \MOF_k(G_1) + \MOF_k(G_2).
\end{equation}
\end{proof}

\subsubsection{Lower bounds on $\MOF_k(G)$} \label{sec: MOF lower bounds}

We now prove some lower bound for the maximum oriented $k$-forcing number of a graph. The first is an application of Proposition \ref{Trivial Lower Bound}

\begin{proposition} \label{prop: MOF >= half min. degree}
For any graph $G$ with minimum degree $\delta$,
\begin{equation}
\nonumber \MOF_k(G)\geq \max \Big \{ \Big \lfloor \frac{\delta}{2} \Big \rfloor - k + 1, 1 \Big \}.
\end{equation}
\end{proposition}

\begin{proof}
Clearly $\MOF_k(G)\geq 1$. Let $D$ be a balanced orientation of $G$ so that, 
\begin{equation}
\nonumber |\deg^+(v) - \deg^-(v)| \leq 1,
\end{equation}
for all vertices $v$ of $D$. In particular, $\delta^+(D) \geq \Big \lfloor \frac{\delta}{2} \Big \rfloor$. Then by Proposition \ref{Trivial Lower Bound}, we have, 
\begin{equation}
\nonumber F_k(D) \geq \max\{\delta^+ - k + 1, 1\} \geq \max \Big \{ \Big \lfloor \frac{\delta}{2} \Big \rfloor - k + 1, 1 \Big \}.
\end{equation}
The desired inequality now follows since $\MOF_k(G) \geq F_k(D)$.
\end{proof}

If $k=1$ and $\delta \geq 2$, then Proposition \ref{prop: MOF >= half min. degree} has the following immediate corollary.

\begin{corollary} \label{cor: MOF and min degree}
For any graph $G$ with minimum degree $\delta \geq 2$, $\MOF(G) \geq \big \lfloor \frac{\delta}{2} \big \rfloor.$
\end{corollary}

If $D$ is an orientation of a graph $G$ and $v$ is any vertex of $D$ with $d^-(v) = 0$ (or equivalently, $d^+(v)=d(v)$), then $v$ must be in every $k$-forcing set for $D$. The result below extends this observation, by orienting all edges away from independent vertex sets. 

\begin{theorem} \label{MOF and independence number}
If $G$ is a graph, then $\MOF_k(G) \geq \alpha(G)$. Moreover, if $k \geq \Delta$, then $\MOF_k(G) = \alpha(G)$.
\end{theorem}

\begin{proof}
Let $I$ be a largest independent set in $G$; that is, $|I|=\alpha(G)$. Next let $D$ be any orientation for $G$ such that $d^-(v) = 0$ for all $v \in I$ (all edges are oriented away from $I$). In this case, every $k$-forcing set for $D$ must contain all of the vertices in $I$, since otherwise no vertices in $I$ could ever become colored during the $k$-forcing process on $D$. Therefore, it follows that $\MOF_k(G) \geq F_k(D) \geq |I| = \alpha(G)$, proving the first assertion.

For the second assertion, suppose $k \geq \Delta$. Let $D$ be an orientation of $G$ realizing $\MOF_k(G)$ and let $S$ be a minimum oriented $k$-forcing set of $D$. Thus, $\MOF_k(G) =|S| = F_k(D)$. We know from the first paragraph that $\alpha(G) \leq |S|$, so it is enough to show now that $\alpha(G) \geq |S|$. If $u,v\in S$ such that $(u,v)$ is an arc in $D$, then remove $v$ from $S$ and consider the set $S'=S-\{v\}$. Color the vertices of $S'$. In the oriented $k$-forcing process starting from $S'$, $u$ will $k$-force $v$ on the first step because $u$ has at most $\Delta$ non-colored out-neighbors and $k \geq \Delta$. This results in a set $T$, which is a superset of $S$, and is therefore a $k$-forcing set as pointed out in Observation \ref{subset forcing}. Hence, $S'$ is an oriented $k$-forcing set of smaller order than $S$, contradicting the fact that $S$ is a minimum oriented $k$-forcing set of $D$. Consequently, no such arc $(u,v)$ can exist in $S$, so no edges can exist between vertices of $S$ in $G$, and we deduce that $S$ is an independent set of $G$. This means that $\alpha(G) \geq |S|$, and completes the proof of the second assertion. 
\end{proof}

Before leaving this section, we investigate lower bounds for the maximum oriented $1$-forcing number in terms of order and average degree of a graph. The following proposition is an exercise in \cite{Bollobas}.

\begin{proposition}\label{bollobas}
Every graph with average degree $d$ has an induced subgraph with minimum degree at least $\frac{d}{2}$.
\end{proposition}

\begin{proposition}
If $G$ is a graph with minimum degree $\delta\ge 2$ and average degree $d$, then
\begin{equation}
\nonumber \MOF(G) \geq \Big \lfloor \frac{d+1}{4} \Big \rfloor.
\end{equation}
\end{proposition}

\begin{proof}
By Proposition \ref{bollobas}, $G$ has an induced subgraph $H$ with $\delta(H) \geq \frac{d}{2}$. Since $\delta(H)$ is an integer, this implies that $\delta(H) \geq \Big \lceil \frac{\lceil d \rceil}{2} \Big \rceil$. Applying Corollary \ref{cor: MOF and min degree}, we have
\begin{equation}\label{eq6}
\MOF(H) \geq \Big \lfloor \frac{\delta(H)}{2} \Big \rfloor \geq \Bigg \lfloor \frac{ \Big \lceil \frac{ \lceil d \rceil}{2} \Big \rceil }{2} \Bigg \rfloor.
\end{equation}
Now, if $\lceil d \rceil$ is even, then
\begin{equation}
\nonumber \Bigg \lfloor \frac{ \Big \lceil \frac{ \lceil d \rceil}{2} \Big \rceil }{2} \Bigg \rfloor = \Bigg \lfloor \frac{\lceil d \rceil }{4} \Bigg \rfloor = \Bigg \lfloor \frac{\lceil d \rceil + 1}{4} \Bigg \rfloor,
\end{equation}
where the last equality holds since $\lceil d \rceil$ is congruent to $0$ or $2$ modulo $4$. On the other hand, if $\lceil d \rceil$ is odd, then
\begin{equation}
\nonumber \Bigg \lfloor \frac{ \Big \lceil \frac{ \lceil d \rceil}{2} \Big \rceil }{2} \Bigg \rfloor = \Bigg \lfloor \frac{\lceil d \rceil + 1}{4} \Bigg \rfloor.
\end{equation}
In either case, it follows from inequality \eqref{eq6} that
\begin{equation}
\MOF(H) \geq \Bigg \lfloor \frac{\lceil d \rceil + 1}{4} \Bigg \rfloor \geq \Bigg \lfloor \frac{d + 1}{4} \Bigg \rfloor.
\end{equation}
The result now follows since, by Proposition \ref{prop: MOF monotonicity}, $\MOF(G) \geq \MOF(H)$.
\end{proof}

It is well known that $\alpha(G) \geq \frac{n}{d + 1}$, which follows from a celebrated theorem of Tur\'{a}n in \cite{Turan}. We make use of this result in our next proposition.

\begin{proposition}
For any graph $G$, $\MOF(G) \geq \big \lfloor \frac{\sqrt{n}}{2} \big \rfloor$.
\end{proposition}

\begin{proof}
Let $d$ be the average degree of $G$. If $d \leq 2\sqrt{n} - 1$, then by Tur\'{a}n's theorem, $\alpha(G) \geq \frac{n}{d+1} \geq \frac{\sqrt{n}}{2}$. Since $\MOF(G) \geq \alpha(G)$, we are done. Otherwise, if $d > 2\sqrt{n} - 1$, then let $H$ be an induced subgraph of $G$ with $\delta(H) \geq \frac{d}{2}$. Let $D$ be any balanced orientation for $H$ -- that is, $|\deg^+(v) - \deg^-(v)| \leq 1$ for all $v \in H$. Then in particular, $\delta^+(D) \geq \lfloor \frac{\delta(H)}{2} \rfloor \geq \lfloor \frac{d}{4} \rfloor$. Since $d > 2\sqrt{n} - 1$, we have $\delta^+(D) > \lfloor \frac{2\sqrt{n} - 1}{4} \rfloor$. Hence $\delta^+(D) \geq \big \lfloor \frac{\sqrt{n}}{2} \big \rfloor$, and it follows from Proposition \ref{prop: MOF monotonicity} that,
\begin{equation}
\nonumber \MOF(G) \geq \MOF(H) \geq F_1(D) \geq \Big \lfloor \frac{\sqrt{n}}{2} \Big \rfloor.
\end{equation}
\end{proof}

%%%%%%%%%%%%%%%%%%%%%%%%%%%%%%%%%%%%%%%%%%%%%%%%%%%%%%%%%%%%%%%%%%%%%%%%%%%%%%%%%%%%%%%%%%%%%%%%%%%%%%%%%%%%%%%%%%%%%%%%%%%

\subsubsection{Upper bounds on $\MOF(G)$}
In this section, we give a general upper bound on the maximum oriented forcing number and then follow it with some applications.

\begin{theorem}\label{upper MOF}
Let $G$ be a graph with $n$ vertices and let $D$ be an orientation of $G$ which realizes $\MOF(G)$, so that $F(D)=\MOF(G)$. If $H$ is an induced subgraph of $D$, then 
\[ \MOF(G)\leq F(H)+n-|H|\leq \MOF(H)+n-|H|. \]
\end{theorem}
\begin{proof}
Suppose $H$ is an induced subgraph of $D$ (an induced subgraph of $G$ with the inherited orientation from $D$). Let $S$ be a minimum forcing set of $H$ of cardinality $F(H)$, and let $Z=(G\setminus H) \cup S$. Clearly, $Z$ is a forcing set of $D$, though it is not necessarily minimum . Thus,
\[
\MOF(G)=F(D)\leq |Z| = |S| + (|G\setminus H|) = F(H) + n -|H| \leq \MOF(H)+n-|H|,
\]
which completes the proof.
\end{proof}

For an application of Theorem \ref{upper MOF}, recall that the clique number of a graph $G$, denoted $\omega(G)$, is the cardinality of a largest complete subgraph of $G$. 

\begin{corollary}
If $G$ is a graph with $n$ vertices and clique number $\omega(G)$. Then, 
\[ \MOF(G) \leq n - \frac{\log_2(\omega(G))}{2}. \]
\end{corollary}
\begin{proof}
Let $C$ denote a largest complete subgraph of $G$ with order $\omega(G)$. By a result of Erdos and Moser in \cite{transitive}, $C$ contains a transitive tournament $H$ on $\log_2(\omega(G))$ vertices. By the corollary above, $\MOF(H)=\frac{\log_2(\omega(G))}{2}$. Thus, from Theorem \ref{upper MOF}, 
\[
\MOF(G) \leq F(H) + n - |H| \leq \MOF(H) + n -|H| = \frac{\log_2(\omega(G))}{2} + n - \log_2(\omega(G)) = n - \frac{\log_2(\omega(G))}{2}.
\]
\end{proof}

A second application of Theorem \ref{upper MOF} is as follows. Recall a matching is a collection of mutually non-incident edges. An induced matching $M$ is a matching where no two edges of $M$ are joined by an edge of $G$. We will denote by $\mim(G)$ the number of edges in a largest induced matching. 

\begin{corollary}\label{mim}
If $G$ is a graph with $n$ vertices, then $\MOF(G) \leq n - \mim(G)$. 
\end{corollary}
\begin{proof}
Suppose $G$ is a graph of order $n$ and let $H$ be a maximum induced matching of $G$. Now, $H$ has $\mim(G)$ isolated edges and $|H|=2\mim(G)$ vertices. Let $D$ be an orientation of $G$ realizing $\MOF(G)$, so that $F(D)=\MOF(G)$.  Now, the induced orientation inherited by $H$ will clearly have forcing number equal to half the order of $H$, or $F(H)=\frac{|H|}{2}=\mim(G)$. Now, applying Theorem \ref{upper MOF} we get,
\[
\MOF(G)=F(D)\leq F(H)+n-|H| = \mim(G)+n-2\mim(G)=n-\mim(G),
\]
which completes the proof.
\end{proof}

%%%%%%%%%%%%%%%%%%%%%%%%%%%%%%%%%%%%%%%%%%%%%%%%%%%%%%%%%%%%%%%%%%%%%%%%%%%%%%%%%%%%%%%%%%%%%%%%%%%%%%%%%%%%%%%%%%%%%%%%%%%

\subsubsection{The maximum oriented $k$-forcing number for Trees}\label{trees}

In this section we consider the maximum oriented $k$-forcing number for trees. Let  $T$ be a tree. It is well known that if $u$ and $v$ are diametric vertices in a tree, then $u$ and $v$ have degree 1. A degree 1 vertex in a tree is called a \emph{leaf}. Since $T$ is a tree, there is exactly one shortest path $P$ from $u$ to $v$ in $T$. The neighbors of $u$ and $v$ on $P$ are called the \emph{stem} of $u$ and $v$, respectively.

\begin{lemma}\label{lem: trees lemma}
Let $T$ be a tree with diameter $\diam(T)\ge 3$. Let $u,v \in V$ be diametric, $P$ the shortest path in $T$ from $v$ to $u$, and $w$ the stem of $u$. Then every neighbor of $w$ is a leaf except one, namely its neighbor $z \neq u$ on the path $P$. Moreover, if $T^*$ is the subgraph of $T$ induced by $V \setminus (N[w] \setminus \{z\})$ and $w$ has $q$ leaf neighbors, then $T^*$ has independence number $\alpha(T^*) \leq \alpha(T) - q$.
\end{lemma}

\begin{proof}
First, observe that $z$ is not a leaf since the diameter of $T$ is at least 3. Observe also that since $u$ and $v$ are diametric, every neighbor of $w$, except $z$, is a leaf. To see that $\alpha(T^*) \leq \alpha(T) - q$, suppose that $\alpha(T^*) > \alpha(T) - q$ and let $I^*$ be a maximum independent set in $T^*$. Then $I = I^* \cup \{y \in N(w) : y \textrm{ is a leaf}\}$ is an independent set in $T$ with $|I| = |I^*| + q > \alpha(T) - q + q = \alpha(T)$, which is impossible.
\end{proof}

Next we show that the maximum oriented $1$-forcing number of any tree is equal to the independence number.

\begin{theorem}\label{thm: MOF(T) = alpha}
If $T$ is a tree, then $\MOF(T) = \alpha(T)$.
\end{theorem}

\begin{proof}
The theorem is trivial when $n(T) \leq 2$, so we assume that $n(T) \geq 3$. Observe that every tree $T$ with diameter at most 2 is a star -- that is, $T$ has a unique maximum degree vertex $v$ of degree $n(T) - 1$ and every other vertex of $T$ is a leaf. Then $\alpha(T) = n(T) - 1$, since the set of leaves of $T$ is a maximum independent set. Thus, by Theorem \ref{MOF and independence number}, $n(T) - 1 = \alpha(T) \leq \MOF(T)$. Clearly $\MOF(T) \leq n(T) - 1$ which means that $\MOF(T) = \alpha(T)$. Thus we may assume that the diameter of $T$ is at least 3, and that $T$ is not a star. This also establishes the base case for induction on the order $n(T)$, since all graphs with $n \leq 3$ have diameter at most 2.

By Theorem \ref{MOF and independence number}, $\MOF(T) \geq \alpha(T)$. We show by induction on the order of $T$ that $\MOF_k(T) \leq \alpha(T)$. Suppose that $\MOF_(T) \leq \alpha(T)$ for all trees $T$ with order $n(T) < p$ for some positive integer $p \geq 4$. Let $T$ be a tree with order $n(T) = p$, $u,v \in V$ diametric, $P$ the shortest path in $T$ from $v$ to $u$, $w$ the stem of $u$ and $z$ the non-leaf neighbor of $w$ on $P$ (which exists by Lemma \ref{lem: trees lemma} since $T$ has diameter at least 3). The the edge $\{w, z\}$ is a bridge in $T$. Let $T_1$ be the subgraph of $T$ induced by $V \setminus (N[w] \setminus \{z\})$ and $T_2$ the subgraph of $T$ induced by $N[w]\setminus \{z\}$. 

Observe that if $w$ has $q$ leaf neighbors then $T_1$ is a tree with order $n(T^*) = n(T) - q - 1 < p$ and that $n(T^*) \geq 2$ since $T$ has diameter at least 3. Let $D$ be an orientation of $T$ for which $f_k(D) = \MOF_k(T)$, and let $D_1, D_2$ be the orientations for $T_1$ and $T_2$, respectively, inherited from $D$. By inductive hypothesis and Lemma \ref{lem: trees lemma}, $\MOF_k(T_1) \leq \alpha(T_1) \leq \alpha(T) - q$. 

Next observe that $T_2$ is a star, and thus by what was shown in the first paragraph, $\MOF_k(T_2) = n(T_2) - 1 = q$ since $T_2$ contains $w$ and the $q$ leaf neighbors of $w$. Thus, by Proposition \ref{bridge},
\begin{equation}
\nonumber \MOF_k(T) \leq \MOF_k(T_1) + \MOF_k(T_2) \leq \alpha(T) - q + q = \alpha(T).
\end{equation}
Thus, $\MOF_k(T)=\alpha(T)$ and the general result now follows by induction.
\end{proof}

In light of Observation \ref{monotonic_MOF} and Theorem \ref{MOF and independence number}, we have the following corollary.

\begin{corollary}
If $T$ is a tree, then $\MOF_k(T) = \alpha(T)$ for all positive integers $k$.
\end{corollary}

%%%%%%%%%%%%%%%%%%%%%%%%%%%%%%%%%%%%%%%%%%%%%%%%%%%%%%%%%%%%%%%%%%%%%%%%%%%%%%%%%%%%%%%%%%%%%%%%%%%%%%%%%%%%%%%%%%%%%%%%%%%

\section{Concluding remarks and open problems}\label{conclusion}
In this paper we have given an extensive study of $k$-forcing in oriented graphs, introduced $\mof_k(G)$ and $\MOF_k(G)$, and related these new graph invariants to a myriad of well studied graph parameters.  However, many interesting problems and questions remain. We highlight that Corollary~\ref{pathcover} states that $\mof(G) = \rho(G)$, where $\rho(G)$ denotes the path covering number of $G$; a well studied graph parameter. This observation suggests we concentrate on the lesser known $\MOF(G)$, and so, we pose the following open problems with this in mind. 
\begin{problem}\label{MOF lower k}
If $G$ is a graph on $n$ vertices and $k$ is a positive integer, determine if $\MOF_k(G) \geq \lceil \frac{n}{k+1} \rceil$.
\end{problem}

Taking $k=1$ in Problem~\ref{MOF lower k}, we next state a very simple and interesting case.
\begin{problem}\label{n/2.1}
For every graph $G$ with order $n \geq 2$, determine if $\MOF(G) \geq \frac{n}{2}$.
\end{problem}

Next recall that the \emph{matching number} of $G$ is typically denoted $\mu(G)$. The following problem would imply Problem~\ref{n/2.1}.
\begin{problem}
For every graph $G$ with order $n\ge 2$, determine if $\MOF(G)\ge n - \mu(G)$.
\end{problem}
It is not hard to see that complete bipartite graphs $K_{x,y}$ with order $n$ have $\MOF(K_{x,y}) = n - 1$. However, a characterization of graphs achieving this equality eludes us. With this in mind we pose the following problem.
\begin{problem}
Characterize connected graphs $G$ with order $n\ge 2$  for which $\MOF(G) = n - 1$.
\end{problem}

Finally, we would like to acknowledge David Amos for his helpful conversations and for some work on earlier drafts of this paper.


\begin{thebibliography}{10}

\bibitem{AIM}
AIM Minimum Rank-Special Graphs~Work Group, Zero forcing sets and the
  minimum rank of graphs, \emph{Linear Algebra Appl.} \textbf{428} (2008), no.~7, 1628--1648.

\bibitem{Amos2014} D.~Amos, Y.~Caro, R.~Davila, and R.~Pepper, Upper bounds on the $k$-forcing number of a graph, \emph{Discrete Applied Mathematics} \textbf{181} (2015), 1--10.

%\bibitem{AzharCaroPepper}
%M. Azhar, Y. Caro and R. Pepper, Maximum oriented forcing number for complete graphs, preprint.

\bibitem{MinRankDigraphs} A.~Berliner, M.~Catral, L.~Hogben, M.~Huynh, K.~Lied, and M.~Young, Minimum rank, maximum nullity, and zero forcing number of simple digraphs, \emph{Electronic Journal of Linear Algebra} \textbf{26}  (2013), Article 52.

\bibitem{Bollobas} B.~Bollob{\'a}s, Modern graph theory, \emph{Graduate Texts in Mathematics}, vol. 184, Springer-Verlag, New York, 1998.

\bibitem{Burgarth} D.~Burgarth and V.~Giovannetti, Full control by locally induced relaxation, \emph{Phys. Rev. Lett.} \textbf{99} (2007), 100501.

%\bibitem{Caro} Y.~Caro, New results on the independence number, Tech. report, Tel-Aviv University, 1979.

\bibitem{Dynamic} Y.~Caro and R.~Pepper, Dynamic approach to $k$-forcing, \emph{Theory and Applications of Graphs}: Vol. 2: Iss. 2, Article 2, (2015).

\bibitem{MOF_Kn} Y. Caro and R. Pepper, Maximum oriented forcing number for complete graphs, manuscript, 2017.

\bibitem{Out-Domination} G.~Chartrand, F.~Harary, and B.Q. Yue, On the out-domination and in-domination numbers of a digraph, \emph{Discrete Mathematics},
  (1999), Vol. 197-198, 179 --183.

\bibitem{Chilakammari} K. Chilakammari, N. Dean, C. X. Kang and E. Yi, Iteration index of a zero forcing set in a graph, \emph{Bull. Inst. Combin. Appl.} \textbf{64} (2012),  57--72.

\bibitem{DaHe17+} R. Davila and M. Henning, \emph{The forcing number of graphs with a given girth, to appear in Quaestiones Mathematicae.}


\bibitem{Davila Kenter} R. Davila and F. Kenter, Bounds for the zero forcing number of a graph with large girth. {\em Theory and Applications of Graphs}, Volume 2, Issue 2, Article 1, 2015.


\bibitem{Dreyer} P. A. Dreyer Jr. and F. S. Roberts, Irreversible $k$-threshold processes: Graph-theoretic threshold models of the spread of disease and of opinion, \emph{Discrete Applied Mathematics} \textbf{157} (2009), 1615--1627.

\bibitem{Edholm} C. J. Edholm, L. Hogben, M. Huynh, J. LaGrange and D. D. Row, Vertex and edge spread of the zero forcing number, maximum nullity, and minimum rank of a graph, \emph{Linear Algebra and its Applications} \textbf{436} (2012), 4352--4372.
	
\bibitem{transitive} P. Erdos and L. Moser, On the representation of directed graphs as unions of orderings, \emph{Publ. Math. Inst. Hungar. Acad. Sci.} \textbf{9} (1964), pp. 125--132.

\bibitem{Gallai-Milgram} T.~Gallai and A.N. Milgram, Verallgemeinerung eines Graphentheoretischen Satzes von R\'edei, \emph{Acta Sc. Math.} \textbf{21} (1960), 181--186.

%\bibitem{Hansberg} A. Hansberg and R. Pepper, On $k$-domination and $j$-independence in graphs, \emph{Discrete Applied Mathematics} \textbf{161} (2013), 1472--1480.

\bibitem{Hogben2010} L.~Hogben, Minimum rank problems, \emph{Linear Algebra and its Applications} 432 (2010), no.~8, 1961 -- 1974, Special issue devoted to the 15th \{ILAS\} Conference at Cancun, Mexico, June 16--20, 2008.

\bibitem{well-balanced} Z. Kiraly and Z. Szigeti, Simultaneous well-balanced orientations of graphs, \emph{Journal of Combinatorial Theory, Series B}, Volume 96, Issue 5, Pages 684-692, 2006.

\bibitem{path-partition} C. Magnant, H. Wang and S. Yuan, Path partitions of almost regular graphs, \emph{Australasian Journal of Combinatorics}, Volume 64 (2) (2016), 334--340.

%\bibitem{Meyer} S. A. Meyer, Zero forcing sets and bipartite circulants, \emph{Linear Algebra and its Applications} \textbf{436} (2012), 888--900.

\bibitem{BalancedOrientation} C.~St. J.~A. Nash-Williams, Well-balanced orientations of finite graphs and unobtrusive odd-vertex-pairings, \emph{Recent Progress in Combinatorics} (Proc. Third Waterloo Conf. on Combinatorics, 1968), Academic Press, New York, 1969, 133--149.

\bibitem{Row2} D. Row, A technique for computing the zero forcing number of a graph with a cut-vertex, \emph{Linear Algebra and its Applications} \textbf{436} (2012), 4423--4432.

\bibitem{Row} D. Row, \emph{Zero forcing number: Results for computation and comparison with other graph parameters}, Ph.D. Thesis, Iowa State University, 2011.

\bibitem{Turan} P.~Turan, On an extremal problem in graph theory, \emph{Math Fiz. Lapok} \textbf{48} (1941), 436--452 (in Hungarian).

%\bibitem{Wei} V.K. Wei, A lower bound on the stability number of a simple graph, \emph{Technical Memorandum} TM 81-11217-9, Bell Laboratories, 1981.

\bibitem{West} D.~West, \emph{Introduction to Graph Theory}, second edition, Prentice Hall Inc., Upper Saddle River, NJ, 2001.

\bibitem{Yi} E. Yi, On the zero forcing number of permutation graphs, \emph{Combinatorial Optimization and Applications}, Lecture Notes in Computer Science Volume 7402, 2012, 61--72.

\bibitem{power} M. Zhao, L. Kang, and G. J. Chang, Power domination in graphs, \emph{Discrete Mathematics} \textbf{306} (2006), 1812--1816.

\end{thebibliography}
\end{document}